\documentclass[12pt]{article} 

\usepackage{amsmath}
\usepackage{amssymb}
\usepackage{amsthm}
\usepackage[all,cmtip]{xy}
\usepackage{fancyhdr}
\usepackage{euscript}
\usepackage{setspace}
\usepackage{graphicx,epsfig}
\usepackage{palatino, url, multicol}
\usepackage{mathrsfs} 
\usepackage{enumerate}   
\usepackage{verbatim} 

\theoremstyle{definition}

\newtheorem{theorem}{Theorem}[section]

\newtheorem{lemma}[theorem]{Lemma}

\newtheorem{proposition}[theorem]{Proposition}

\newtheorem{corollary}[theorem]{Corollary}

\newtheorem{remark}[theorem]{Remark}

\newcommand{\IR}{\hbox{$\mathbb{R}$}}

\newcommand{\IZ}{\hbox{$\mathbb{Z}$}}

\newcommand{\IQ}{\hbox{$\mathbb{Q}$}}

\newcommand{\abs}[1]{\hbox{$\left| {#1} \right|$}}

\def\tt{\texttt}
\def\bf{\textbf}

\def\IE{\hbox {I \hskip -5.2pt {E}}}

\def\IR{\mathbb{R}}
\def\IQ{\mathbb{Q}}

\def\I1{\mathbb{1}}

\def\limx0{\lim_{x \to 0}}

\def\intxyleq1{\underset{\| x - y  \| \leq 1}{\int}}
\def\intxygeq1{\underset{\| x - y  \| \geq 1}{\int}}
\def\intxizetaleq1{\underset{\| \xi - \zeta  \| \leq 1}{\int}}
\def\intxizetageq1{\underset{\| \xi - \zeta \| \geq 1}{\int}}

\def\tab{\hskip 1mm}

\def\tab{\hspace{.1pc}}
\def\ttab{\hspace{1pc}}

\newcounter{hours}
\newcounter{minutes}
\newcommand\printtime{%
  \setcounter{hours}{\the\time/60}%
  \setcounter{minutes}{\the\time-\value{hours}*60}%
  \ifthenelse{\value{hours} > 12}
     {
       \setcounter{hours}{\value{hours}-12}%
       \thehours:\theminutes \ p.m.                
     }
     {
       \thehours:\theminutes \ a.m.                
     } 
}

\def\putdate{{\tt Compiled on \the\month-\the\day-\the\year \ at\printtime} \\}

\begin{document} 

\title{Bounding differences in Jager Pairs}
 \author{Avraham Bourla\\
   Department of Mathematics\\
   Saint Mary's College of Maryland\\
   Saint Mary's City, MD, 20686\\
   \texttt{abourla@smcm.edu}}
 \date{\today}
 \maketitle

\begin{abstract}
\noindent Symmetrical subdivisions in the space of Jager Pairs for continued fractions-like expansions will provide us with bounds on their difference. Results will also apply to the classical regular and backwards continued fractions expansions, which are realized as special cases.
\end{abstract}

\section{Introduction}

Given a real number $r$ and a rational number, written as the unique quotient $\frac{p}{q}$ of the relatively prime integers $p$ and $q>0$, our fundamental object of interest from diophantine approximation is the \bf{approximation coefficient} $\theta(r,\frac{p}{q}) := q^2\abs{r-\frac{p}{q}}$.  Small approximation coefficients suggest high quality approximations, combining accuracy with simplicity. For instance, the error in approximating $\pi$ using $\frac{355}{113}=3.1415920353982$ is smaller than the error of its decimal expansion to the fifth digit $3.14159 = \frac{314159}{100000}$. Since the former rational also has a much smaller denominator, it is of far greater quality than the latter. Indeed $\theta\big(\pi,\frac{355}{113}\big)< 0.0341$ whereas $\theta\big(\pi,\frac{314159}{100000}\big)> 26535$.\\

\noindent Since adding integers to fractions does not change their denominators, we have $\theta(r,\frac{p}{q}) = \theta(x ,\frac{p}{q} - \lfloor r \rfloor)$, where $\lfloor r \rfloor$ is the largest integer smaller than or equal to $r$ (a.k.a. the the floor of $r$) and $x := r - \lfloor r \rfloor \in [0,1)$, allowing us to restrict our attention to the unit interval. Expanding an irrational initial seed $x_0 \in (0,1) - \IQ$ as an infinite regular continued fraction
\begin{equation}\label{classical_gauss}
x_0 = [b_1,b_2,b_3,...]_0 := \frac{1}{b_1 + \frac{1}{b_2 + \frac{1}{b_3 + ...}}}
\end{equation}
provides us with the unique symbolic representation of the initial seed via the sequence $\{b_n\}_1^\infty$ of positive integers, known as the partial quotients or \bf{digits of expansion} of $x_0$. For all $n \ge 1$, we label the approximation coefficient associated with the $n$th \bf{convergent} $\frac{p_n}{q_n} := [b_1.b_2,...]_0$ of $x_0$ by $\theta_n$ and refer to the sequence $\{\theta_n\}_1^\infty$ as the \bf{sequence of approximation coefficients}. The value of $\displaystyle{\liminf_{n \to \infty}}\{\theta_n(x_0)\}$ measures how well can $x_0$ be approximated by rational numbers, leading to the construction of the Lagrange Spectrum, see \cite{CF}. In 1978, Jurkat and Peyerimhoff \cite{JP} showed that the \bf{Space of Jager Pairs} 
\[\left\{\left(\theta_n(x_0),\theta_{n+1}(x_0)\right) \subset \IR^2, \tab x_0 \in (0,1) - \IQ, \tab n \ge 1 \right\}\] 
is a dense subset of the region in the cartesian plane which is the interior of the triangle with vertices $(0,0), \tab (0,1)$ and $(1,0)$. We may also expand $x_0$ as an infinite backward continued fraction
\begin{equation}\label{classical_renyi}
x_0 = [b_1,b_2,b_3,...]_1 := 1 - \frac{1}{b_1 + 1 - \frac{1}{b_2 + 1 - \frac{1}{b_3 + ...}}}
\end{equation}
yielding a new unique sequence of digits, hence new sequences of convergents and approximation coefficients as well as a new space of Jager pairs.\\

\noindent We are going to reveal an elegant symmetrical internal structure for both spaces of Jager Spaces corresponding to the regular and backwards continued fractions expansions. Our approach is to treat the regular and continued fraction expansions as limiting cases for the two families of one parameter continued fraction-like expansions, first introduced by Haas and Molnar. Using simple plane geometry, we will provide in corollary \ref{result} upper and lower bounds for the growth rate of the associated sequence of approximation coefficients. For instance, knowing a priori that $b_2 = b_3 =1$ and $b_4 = 3$ are the digits of the classical regular continued fraction expansion will allow us to obtain the bounds $\abs{\theta_2-\theta_1} < \frac{\sqrt{2}}{3}$ and $\frac{2\sqrt{2}}{7} < \abs{\theta_3-\theta_2} < \frac{3\sqrt{2}}{5}$.

\begin{remark}
In this paper, we will only concern ourselves with bounding the growth rate of the sequence of approximation coefficients. For the arithmetical properties of this sequence, including its essential bounds and a Borel-type generalization on the minimum of three consecutive terms, refer to \cite{Bourla} 
\end{remark}

\section{Preliminaries}

This section is a paraphrased summery of excerpts from \cite{HM, HM2}, given for sake of completeness. In general, the fractional part of M$\operatorname{\ddot{o}}$bius transformation which map $[0,1]$ onto $[0,\infty]$ leads to expansion of real numbers as continued fractions. To characterize all these transformations, we recall that M$\operatorname{\ddot{o}}$bius transformations are uniquely determined by their values for three distinct points. Thus, we will need to introduce a parameter for the image of an additional point besides 0 and 1, which we will naturally take to be $\infty$. Since our maps fix the real line, the image of $\infty$, denoted by $-k$, can take any value within the set of negative real numbers. We let $m \in \{0,1\}$ equal zero for orientation reversing transformations, i.e. $0 \mapsto \infty, \tab 1 \mapsto 0$ and let $m$ equal one for orientation preserving transformations, i.e. $0 \mapsto 0, \tab 1 \mapsto \infty$. Conclude that all such transformations must be the extension of the maps $x \mapsto \frac{k(1-m-x)}{x-m}, \tab k>0$, mapping $(0,1)$ homeomorphically to $(0,\infty)$, to the extended complex plane. The maps $T_{(m,k)}: [0,1) \to [0,1), \tab 0 \mapsto 0$,
\[T_{(m,k)}(x) = \frac{k(1-m-x)}{x-m} - \bigg\lfloor\frac{k(1-m-x)}{x-m}\bigg\rfloor,\ttab x>0\]
are called \bf{Gauss-like} and \bf{R\'enyi-like} for $m=0$ and $m=1$ respectively.\\ 

\noindent We expand the initial seed $x_0 \in (0,1)$ as an $(m,k)-$continued fraction using the following iteration process: 
\begin{enumerate}
\item Set $n :=1$.
\item If $x_{n-1} = 0$, write the $(m,k)-$expansion of $x_0$ as $[a_1,...,a_{n-1}]_{(m,k)}$ and exit. 
\item Set  the \bf{digit} and \bf {future} of $x_0$ at time $n$ to be $a_n := \big\lfloor  \frac{k(1-m-x_{n-1})}{x_{n-1}-m} \big\rfloor \in \IZ_{\ge 0}$ and $x_n := \frac{k(1-m-x_{n-1})}{x_{n-1}-m} - a_n \in [0,1)$. Increase $n$ by one and go to step 2. 
\end{enumerate}
For all $n \ge 0$, we thus have
\[x_{n+1} = T_{(m,k)}(x_n) =  \frac{k(1-m-x_n)}{x_n-m} - a_n\]
so that
\[x_n = m + \frac{k(1-2m)}{a_{n+1}+k+x_{n+1}}. \] 
Therefore, this iteration scheme leads to the expansion of the initial seed $x_0$ as
\[x_0 = m + \frac{k(1-2m)}{a_1+k+x_1} = m + \frac{k(1-2m)}{a_1 + k + m + \frac{k(1-2m)}{a_2+k+x_2}} = ...\]
\begin{remark}\label{digit_remark}
The classical $k=1$ cases lead to the regular and backwards continued fractions expansions for $m=0$ and $m=1$ respectively. However, due to the definition of the map $T_{(m,k)}$ the digits of expansion will be smaller by one than their classical representation. For instance, 
\[[0,1,2]_{(0,k)} = \dfrac{k}{0+k+\dfrac{k}{1+k+\dfrac{k}{2+k}}} = \frac{k^2 + 4k + 2}{k^2 + 5k + 4}\]
and
\[[0,1,2]_{(1,k)} = 1 - \dfrac{k}{0+k+1-\dfrac{k}{1+k+1-\dfrac{k}{2+k}}} = \frac{k+4}{k^3 + 3k^2 + 5k + 4}\]
will yield, after plugging $k=1$, the fractions $[1,2,3]_0 = \frac{7}{10}$ and $[1,2,3]_1 = \frac{5}{13}$ using the classical expansions \eqref{classical_gauss} and \eqref{classical_renyi}. In order to avoid this confusion, we denote the $n$th $(m,k)-$ digit of expansion by $a_n$, whereas the $n$th $m$-digit of expansion is denoted by $b_n = a_n+1$.  
\end{remark} 

\noindent For the classical regular and backwards continued fraction expansions this iteration process eventually terminates precisely when the initial seed $x_0$ is a rational number. Analogously, we denote the countable set of all numbers in the interval with finite $(m,k)-$expansion by $\IQ_{(m,k)}$. For all $a \in \IZ_{\ge 0}$, the cylinder set 
\begin{equation}\label{Delta}
\Delta_a := \left(\frac{(1-m)k+m{a}}{a+k+1-m}, \frac{(1-m)k+m(a+1)}{a+k+m} \right)
\end{equation}
is defined such that $x_0 \in \Delta_a$ if and only if $a_1=a$. More generally, we have
\begin{equation}\label{x_n}
x_n \in \Delta_a \iff a_{n+1}=a, \ttab n \ge 0,
\end{equation}
that is, the restriction of the map $T_{(m,k)}$ to $\Delta_a$ is a homeomorphism onto (0,1). When $x_0$ is am $(m,k)-$ irrational, Define the \bf{past} of $x_0$ at time $n \ge 1$ to be $Y_1 := m - k - [a_1]_{(m,k)}$ and
\begin{equation}\label{Y_n}
Y_n := m - k - a_n - [a_{n-1},...,a_1]_{(m,k)} \in (m-k-a_n-1, m-k-a_n), \ttab n \ge 2.
\end{equation}
Then for all initial seeds $x_0 \in (0,1) - \IQ_{(m,k)}$ and $n \ge 1$, we have $(x_n,Y_n) \in \Omega_{(m,k)} := (0,1) \times (-\infty, m-k)$. We call the set $\Omega_{(m,k)}$ the \bf{space of dynamic pairs}.\\

\noindent The sequence of approximation coefficients $\left\{\theta_n(x_0)\right\}_1^\infty$ for the $(m,k)-$expansion is defined just like the classical object, i.e. $\theta_n(x_0) := q_n^2\abs{x_0 - \frac{p_n}{q_n}}$, where $\frac{p_n}{q_n} = [a_1,...,a_n]_{(m,k)} \in \IQ_{(m,k)}$ are the appropriate convergents. The sequence of approximation coefficients relates to the future and past sequences of $x_0$ using the identity  
\[\theta_n(x_0) = \dfrac{1}{x_{n+1} - Y_{n+1}},\]
first proven for the classical regular continued fraction case in 1921 by Perron \cite{Perron}. Each of the continuous maps
\begin{equation}\label{Psi}
\Psi_{(m,k)}: \Omega_{(m,k)} \to \IR^2, \hspace{1pc} (x,y) \mapsto \bigg(\dfrac{1}{x-y}, \frac{(m-x)(m-y)}{(2m-1)k(x-y)}\bigg),
\end{equation}
\[(\theta_n, \theta_{n+1}) = \Psi_{(m,k)}(x_{n+1},Y_{n+1}),\]
 is a homeomorphism onto its image 
\[\Gamma_{(m,k)} := \Psi_{(m,k)}(\Omega_{(m,k)}).\] 
Furthermore, the two families $\{\Psi_{(0,k)}\}$ and $\{\Psi_{(1,k)}\}$ are known to be equicontinuous. The corresponding \bf{Space of Jager Pairs} 
\[\{\left(\theta_n(x_0),\theta_{n+1}(x_0)\right) : \tab x_0 \in (0,1) - \IQ_{(m,k)}, \tab n \ge 1\}\] 
is a dense subset of $\Gamma_{(m,k)}$. In our approach, we will require the parameter $k$ to be greater than one, treating the classical regular and backwards continued fraction expansions as the limit of the $(m,k)$-expansions. As $k \to 1^+$. The choice of taking the limit from the right is not arbitrary, for the treatment of the $0 < k < 1$ cases exhibits certain pathologies, see \cite[section 3.3]{Thesis}. 

\section{The finer structure for the space of Jager Pairs}

Fix $m \in \{0,1\}, \tab k \in (1,\infty)$, and an initial seed $x_0 \in (0,1) - \IQ_{(m,k)}$. In order to ease the notation, we will omit the subscripts $\square_{(m,k)}$ from now on. For all $a \in \IZ_{\ge 0}$, define define the regions 
\[P_{(m,k,a)} = P_a := (0,1) \times(m-a-k-1,m-a-k)\] 
and 
\[F_{(m,k,a)} = F_a := \Delta_a \times (-\infty,m-k),\] 
where $\Delta_a$ is the cylinder set \eqref{Delta}. Using the identity \eqref{x_n} and the definition \eqref{Y_n} of $Y_n$, we see that that for all $n \ge 1$ we have
\[\left(x_{n+1},y_{n+1}\right) \in P_a \cap F_b \iff a_{n+1} = a \text{ and  $a_{n+2} = b$}.\] 
We label the images of each of these subsets of $\Omega$ under $\Psi$ by $P_a^\#$ and $F_a^\#$. Using formula \eqref{Psi} and the fact that $\Psi$ is a homeomorphism, we see that for all $n \ge 1$, we have
\begin{equation}\label{theta_a}
\left(\theta_n,\theta_{n+1}\right) \in P_a^\# \cap F_b^\# \iff a_{n+1} = a \text{ and  $a_{n+2} = b$}.
\end{equation} 
Next, let $p_{(m,k,a)}=p_a$ be the open horizontal line segment $(0,1) \times \{m-a-k\}$, let $f_{(m,k,a)}=f_a$ be the open vertical ray $\left\{\frac{(1-m)k+m}{a+k}\right\} \times (-\infty,m-k)$ and let $p_a^\#$ and $f_a^\#$ be their image under $\Psi$. Since both the collections $\left\{P_a \cup p_a \right\}_{a \in \IZ_{\ge 0}}$ and $\left\{F_a \cup f_a \right\}_{a \in \IZ_{\ge 0}}$ partition $\Omega$, the image of their intersections under $\Psi$, $\left\{(P_a^\# \cup p_a^\#) \cap (F_b^\# \cup f_b^\#)\right\}_{a,b \in \IZ_{\ge 0}}$ will partition $\Gamma$. We will call each member of this refined partition a \bf{subdivision}. 
\begin{proposition}\label{P_a_cap_F_b}
For all $k\in \IR >1$ and $a,b \in \IZ^+$ the region $P_a^\# \cap F_b^\#$ is the open interior of the quadrangle in $\IE^2$ with vertices
\[\left(\frac{b+k}{(1-2m)k+(a+k)(b+k)}, \frac{a+k}{(1-2m)k + (a+k)(b+k)}\right),\]
\[\left(\frac{b+k}{(1-2m)k+(a+k+1)(b+k)},\frac{a+k+1}{(1-2m)k + (a+k+1)(b+k)}\right),\]
\[\left(\frac{b+k+1}{(1-2m)k+(a+k)(b+k+1)},\frac{a+k}{(1-2m)k + (a+k)(b+k+1)}\right)\]
and
\[\left(\frac{b+k+1}{(1-2m)k+(a+k+1)(b+k+1)},\frac{a+k+1}{(1-2m)k + (a+k+1)(b+k+1)}\right).\]
\end{proposition}
\begin{proof}
Fix $a \in \IZ^+$ and for every $(x,y_0) \in p_a$ i.e. $x \in (0,1), \tab y_0 := m-k-a$, set $(u,v) := \Psi(x,y_0)$. Using the definition \eqref{Psi} of $\Psi$, we write $u = \frac{1}{x-y_0} = \frac{1}{x+a+k-m}$. Then, as $x$ tends from 0 to 1, $u$ tends from $\frac{1}{a+k-m}$ to $\frac{1}{a+k+1-m}$. Since $m - x = m - y - \frac{1}{u}$, we express $v$ in terms of $u$ as
\[v = \frac{(m-x)(m-y_0)}{(2m-1)k(x-y_0)} = \frac{(m-x)(m-y_0)u}{(2m-1)k} = \frac{a+k}{(2m-1)k}\left((a+k)u - 1\right)\]
so that as $x$ tends from 0 to 1, $v$ tends from $\frac{m(a+k)}{k(a+k-m)}$ to $\frac{(1-m)(k+a)}{k(k+a+1-m)}$. Since $\Psi:\Omega \to \Gamma$ is a homeomorphism, we see that $p_a^\#$ is an open segment of the line $(a+k)^2{u} + (1-2m)k{v} = a+k$ between the points $\left(\frac{1}{a+k-m}, \frac{m(a+k)}{k(a+k-m)}\right)$ and $\left(\frac{1}{a+k+1-m},\frac{(1-m)(k+a)}{k(a+k+1-m)}\right)$.\\

\begin{multicols}{2} 
\begin{center}
\includegraphics[scale=.42]{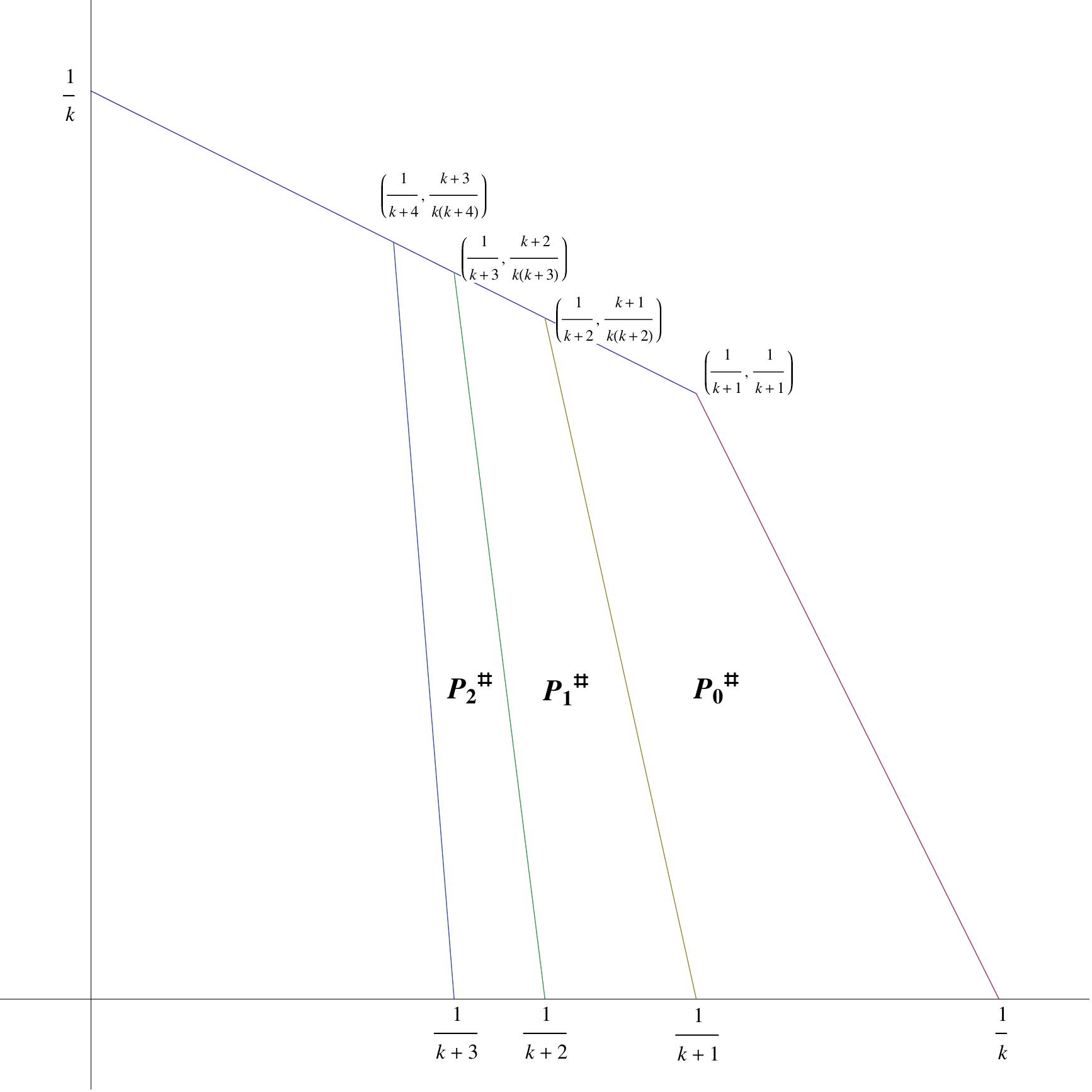}\\
\end{center}
\begin{center}
\includegraphics[scale=.4]{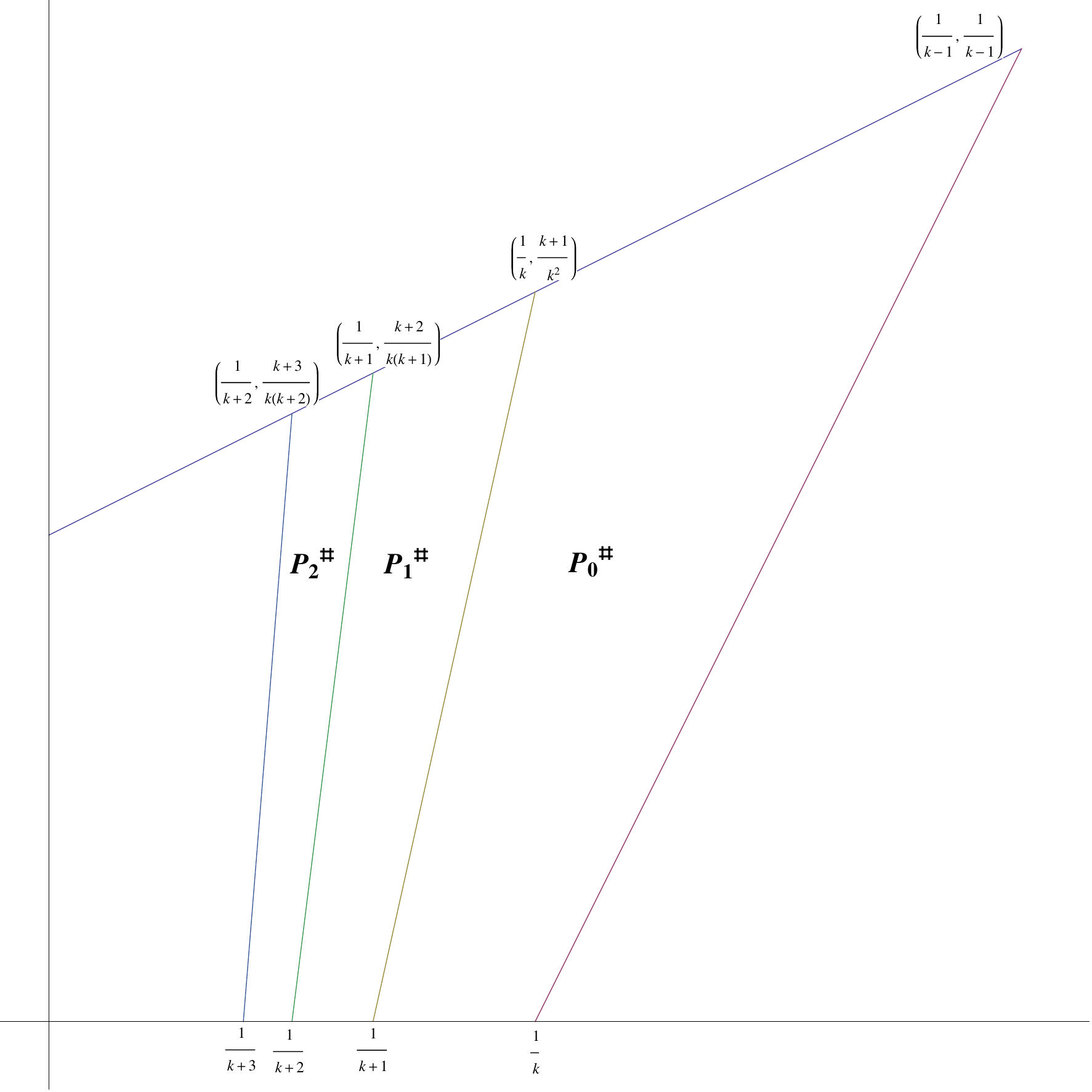}\\
\end{center}
\end{multicols}
\vspace{-1pc}
\hspace{7pc}$\Gamma_{(0,k)}$ \hspace{17pc} $\Gamma_{(1,k)}$\\

\noindent Next, fix $b \in \IZ^+$ and set $x_0 := \frac{(1-m)k+m{b}}{b+k}$. We use the definition \eqref{Psi} of $\Psi$ to first write $m - y = m - x + \frac{1}{u}$ and then express $v$ in terms of $u$ as
\begin{equation}\label{v=f(u)}
v = \frac{u}{(2m-1)k}(m-x_0)\left(m-x_0+\frac{1}{u}\right) = \frac{1}{b+k}\left(\frac{(2m-1)k}{b+k}u+1\right).
\end{equation}
Since $\Psi$ is a homeomorphism, this allows us to conclude that $\Psi$ maps $f_b$ to an open segment of a line in the $u{v}$ plane, which is the reflection of the line $p_b^\#$ along the diagonal $u-v=0$. In particular, $P_a^\# \cap F_b^\#$ is the region interior to the quadrangle with vertices $p_a^\# \cap f_b^\#, \tab p_{a+1}^\# \cap f_b^\#, \tab p_a^\# \cap f_{b+1}^\#$ and $p_{a+1}^\# \cap f_{b+1}^\#$. But $(x_0,y_0) = p_a \cap f_b$, so that $(u_0,v_0) := \Psi(x_0,y_0) = p_a^\# \cap f_b^\#$. The definition \eqref{Psi} of $\Psi$ and  formula \eqref{v=f(u)} now yield  
\[u_0 = \frac{1}{x_0-y_0} = \frac{b+k}{(1-2m)k + (a+k)(b+k)}\]
and
\[v_0 = \frac{a+k}{(1-2m)k + (a+k)(b+k)}\]
as desired. 
\end{proof}

\newpage

\begin{multicols}{2} 
\begin{center}
\includegraphics[scale=.47]{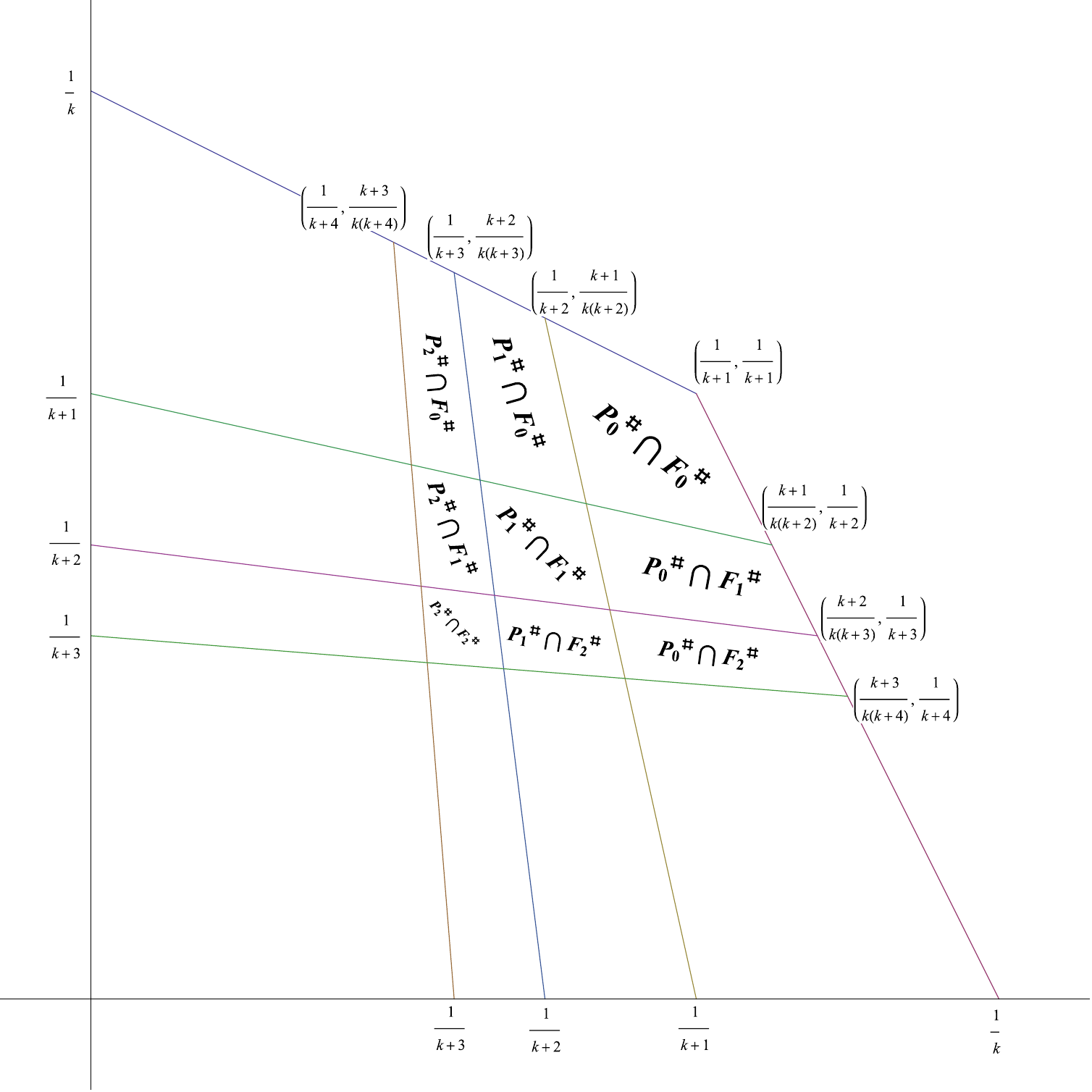}
\end{center}
\begin{center}
\includegraphics[scale=.48]{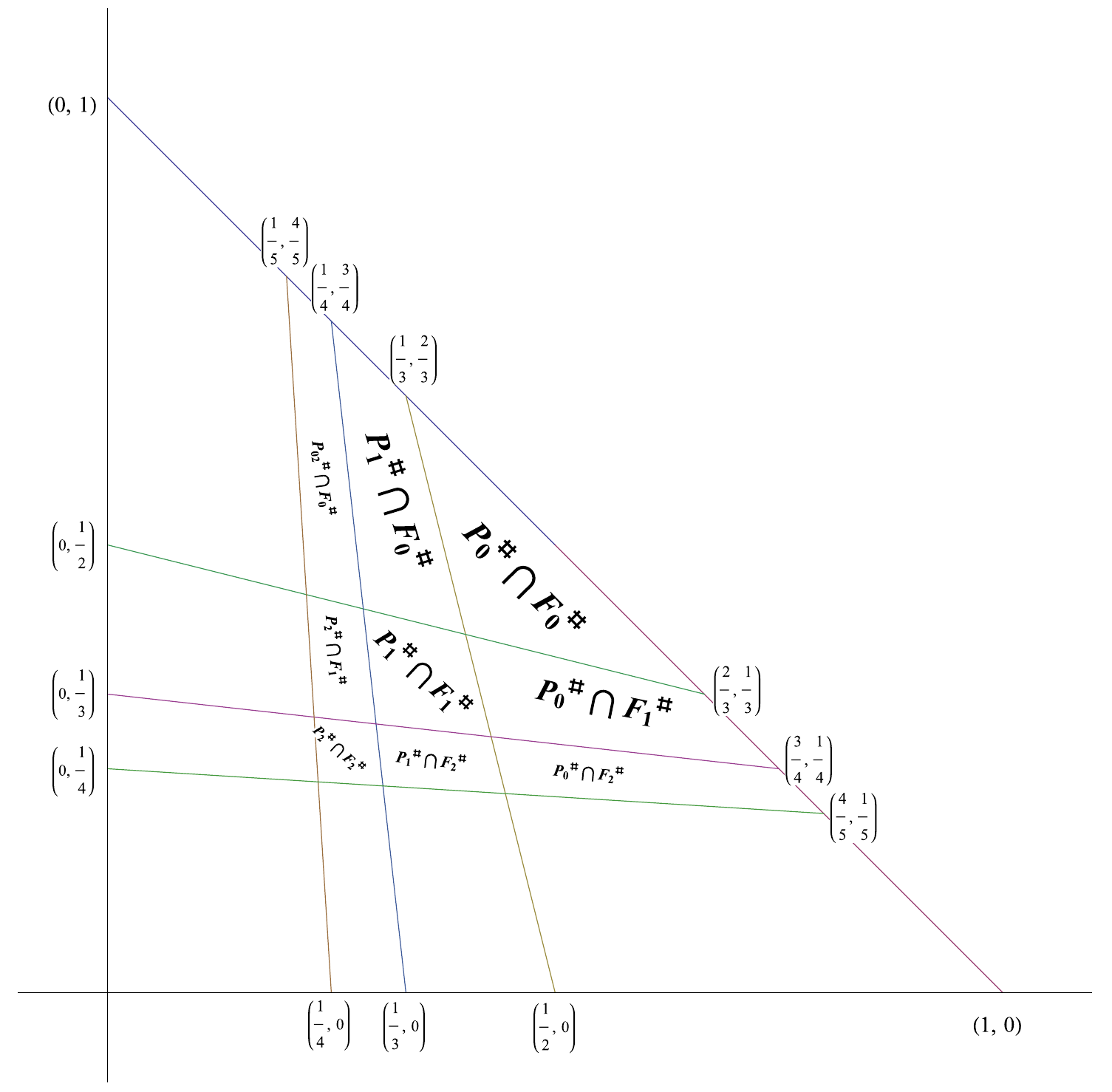}\\
\end{center}
\end{multicols}
\vspace{-1pc}
\hspace{5pc}$\Gamma_{(0,k)}$ revisited \hspace{16pc} $\Gamma_{(0,1)}$\\

\begin{remark}
This result was first proved for the classical regular continued fractions expansion, corresponding to the parameters $m=0$ and $k=1$ by Jager and Kraaikamp \cite{JK} (see also \cite[exercise 5.3.4]{DK}). For this case, the region $P_0^\# \cap F_0^\#$ degenerates to the interior of the triangle with vertices $\left(\frac{1}{3},\frac{2}{3}\right), \tab \left(\frac{2}{3},\frac{1}{3}\right)$ and $\left(\frac{2}{5},\frac{2}{5}\right)$. In these papers, the regions $P_a^\#$ and $F_a^\#$ are denoted by $V_b^*$ and $H_b^*$ respectively, where $b=a+1$. We diverge from their notational choice to better illustrate the dynamical structure at hand: the point $(\theta_{n-1},\theta_n) \in P_a^\#$ precisely when the first digit $a_n$ in $Y_n$, the past of $x_0$ at time $n$, is $a$. Similarly, the point $(\theta_{n-1},\theta_n) \in F_a^\#$ precisely when the first digit $a_{n+1}$ in $x_n$, the future representation of $x_0$ at time $n$, is $a$.   
\end{remark}

\noindent   When $m=1$, we obtain the depicted Spaces of Jager Pairs:
 
\begin{multicols}{2} 
\begin{center}
\includegraphics[scale=.47]{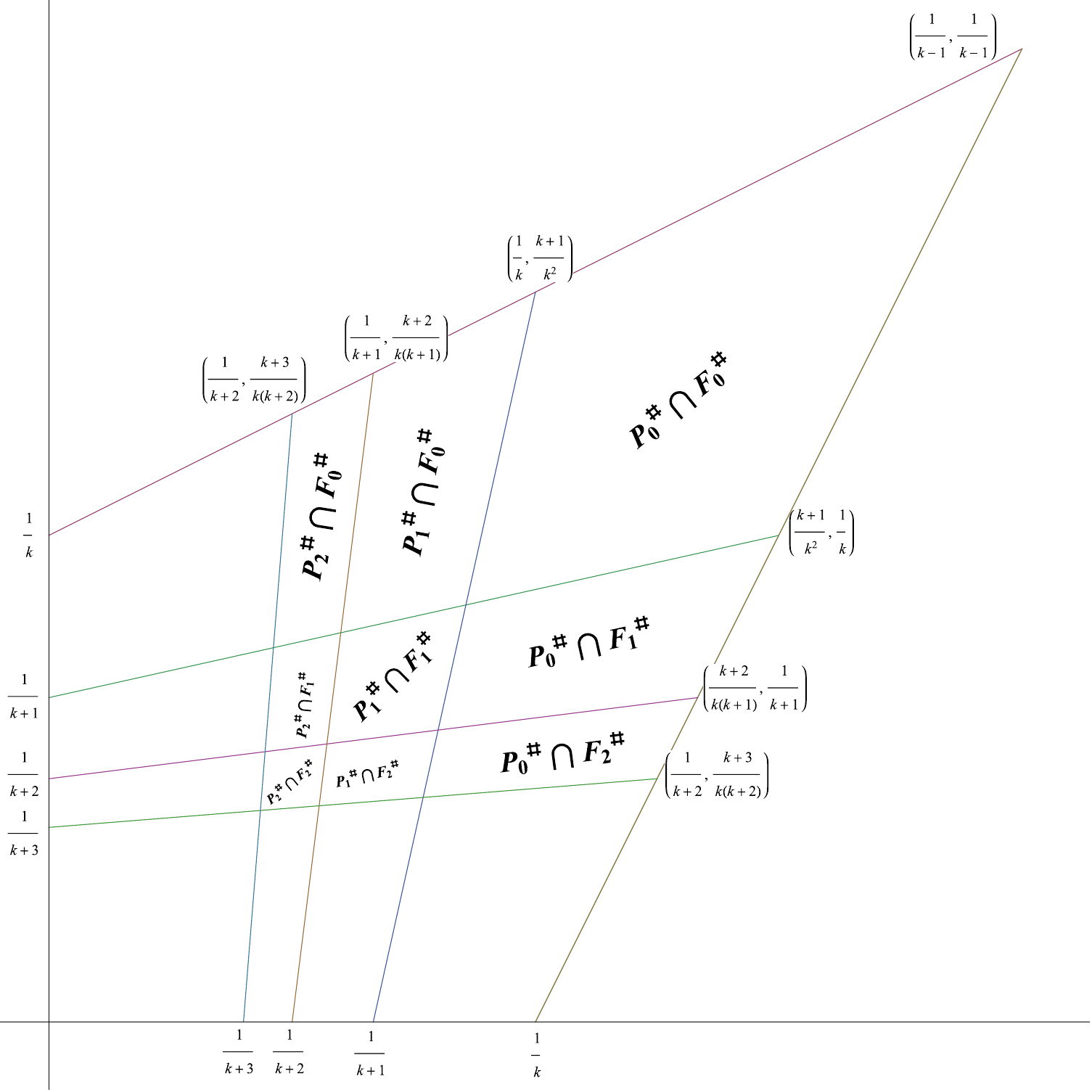}
\end{center}
\begin{center}
\includegraphics[scale=.48]{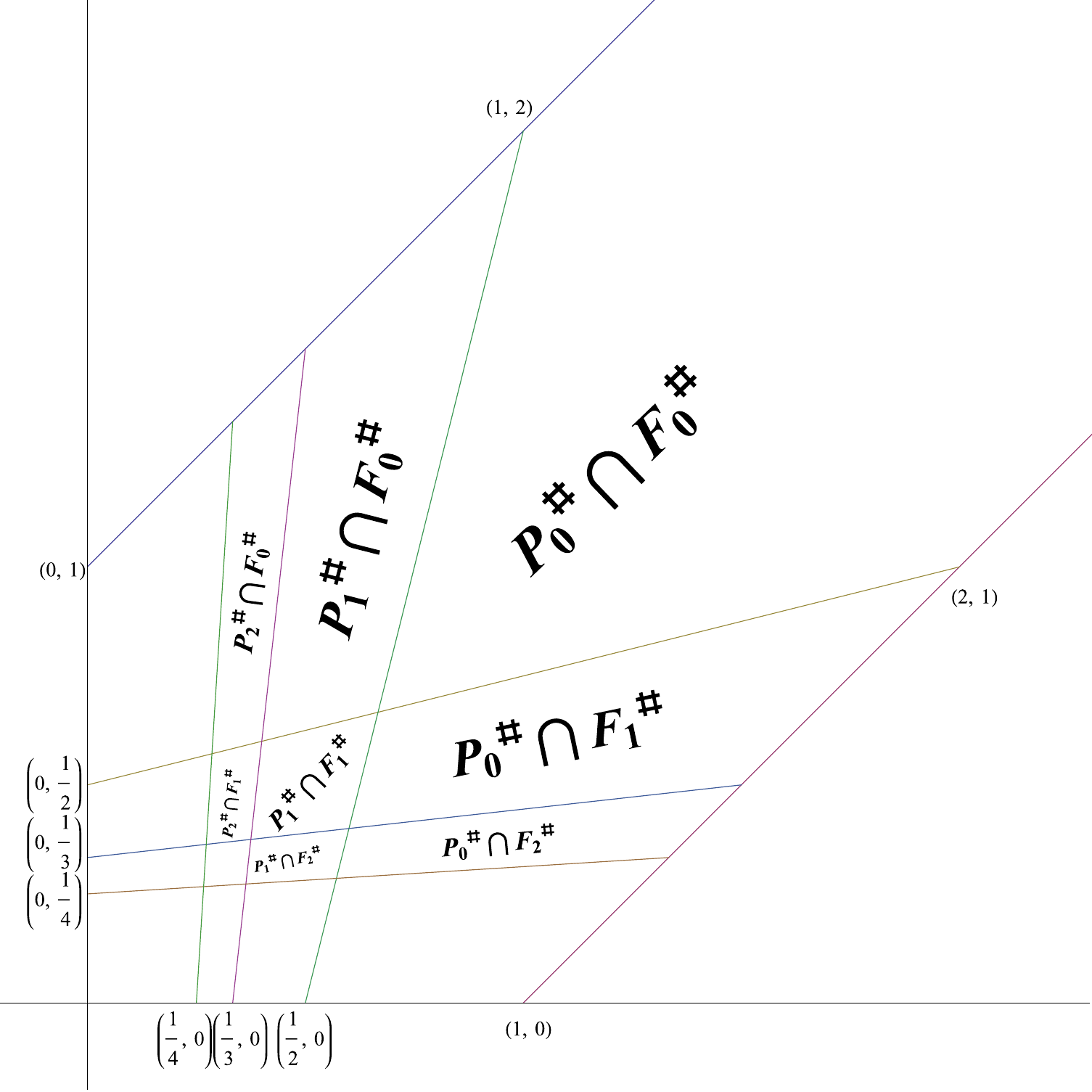}
\end{center}
\end{multicols}
\vspace{-1pc}
\hspace{5pc}$\Gamma_{(1,k)}$ revisited \hspace{16pc} $\Gamma_{(1,1)}$\\

\noindent The picture for the classical backwards continued fractions $k=1$, is obtained after we let $k \to 1^+$ and use the equicontinuity of the family $\{\Psi_{(1,k)}\}_{k>1}$. For this case, the region $P_0^\# \cap F_0^\#$ expands to the unbounded region in the $u{v}$ plane, which is the intersection of the regions $u-v>1, \tab v-u<1, \tab  4u-v>2$ and $4v-u>2$.

\section{Bounding growth rate in Jager Pairs} 

This symmetry in the subdivisions will allow us to provide uniform bounds for the rate of growth of the sequence $\{\theta_n\}_1^\infty$, which is the sequence whose members are the Jager Pair differences $\{\abs{\theta_{n+1}-\theta_n}\}_1^\infty$. But first, we already know that $\Gamma_{(1,1)}$ is the unbounded region in the first quadrant of the $u{v}$ plane bounded by the lines $u-v < 1$ and  $v-u < 1$, hence we obtain at once:
\begin{theorem}
In general, the sequence of approximation coefficients associated with the classical backwards continued fraction expansion $m=k=1$ has no uniform upper bound. However its growth rate is uniformly bounded by 1.
\end{theorem}
\noindent Our main result is:
\begin{theorem}\label{main}
Given $m \in \{0,1\}, \tab k \in \IR > 1, \tab x_0 \in (0,1) - \IQ_{(m,k)}$ and $N \ge 1$, we write $l := \min\{a_{N+1},a_{N+2},a_{N+3}\}$ and $L :=  \max\{a_{N+1},a_{N+2},a_{N+3}\}$, where $a_n$ is the digit at time $n$ in the $(m,k)-$expansion for $x_0$. Then the inequality
\[\left(\theta_{N+1} - \theta_N\right)^2 + \left(\theta_{N+2} - \theta_{N+1}\right)^2 < 2\tab\left(\frac{L-l+1}{(1-2m)k + (l+k)(L+k+1)}\right)^2\]
is sharp. Furthermore, if $L-l>1$, then the inequality
\[\left(\theta_{N+1} - \theta_N\right)^2 + \left(\theta_{N+2} - \theta_{N+1}\right)^2 > 2\tab\left(\frac{L-l}{(1-2m)k + (l+k+1)(L+k)}\right)^2\]
is sharp.
\end{theorem}
\noindent In order to prove this theorem, we will first prove:
\begin{lemma}
Given $m \in \{0,1\}, \tab k \in \IR > 1, \tab x_0 \in (0,1) - \IQ_{(m,k)}$ and $N \ge 1$, write $a := \min\{a_{N+1},a_{N+2}\}, \tab A :=  \max\{a_{N+1},a_{N+2}\}, \tab b :=  \min\{a_{N+2},a_{N+3}\}$ and $B :=  \max\{a_{N+2},a_{N+3}\}$. Then the inequality
\[\left(\theta_{N+1} - \theta_N\right)^2 + \left(\theta_{N+2} - \theta_{N+1}\right)^2\]
\[<  \left(\frac{B+k+1}{(1-2m)k + (a+k)(B+k+1)} - \frac{b+k}{(1-2m)k + (b+k)(A+k+1)}\right)^2\]
\[ + \left(\frac{A+k+1}{(1-2m)k + (b+k)(A+k+1)}- \frac{a+k}{(1-2m)k + (a+k)(B+k+1)}\right)^2\] 
is sharp. Furthermore, if $\max\{A-a,B-b\}>1$, then the inequality 
\[\left(\theta_{N+1} - \theta_N\right)^2 + \left(\theta_{N+2} - \theta_{N+1}\right)^2\]
\[>  \left(\frac{B+k}{(1-2m)k + (a+k+1)(B+k)} - \frac{b+k+1}{(1-2m)k + (b+k+1)(A+k)}\right)^2\]
\[ + \left(\frac{A+k}{(1-2m)k + (b+k+1)(A+k)}- \frac{a+k+1}{(1-2m)k + (a+k+1)(B+k)}\right)^2\]
is sharp. 
\end{lemma}
\begin{proof}
Let $R_{(m,k,a,b,A,B)} = R$ be the set
\[\bigcup_{a \le i \le A, \tab b \le j \le B}(P_i^\# \cap F_j^\#).\] 
Formula \eqref{theta_a} yields $(\theta_N,\theta_{N+1}) \in P_{a_N+1}^\# \cap F_{a_{N+2}}^\# \subset R$ and $(\theta_{N+1},\theta_{N+2}) \in P_{a_{N+2}}^\# \cap F_{a_{N+3}}^\# \subset R$. From proposition \ref{P_a_cap_F_b}, we see that $R$ is the interior of the bounded convex quadrangle, consisting of sections from the line segments $p_a^\#, \tab p_{A+1}^\#, \tab f_b^\#$ and $f_{B+1}^\#$. Unless $\max\{A-a,B-b\} \le 1$, we also let $r_{(m,k,a,b,A,B)} = r$ be the non-empty set
\[\bigcup_{a+1 \le i \le A-1, \tab b+1 \le j \le B-1}(P_i^\# \cap F_j^\#),\]
so that 
\begin{equation}\label{R-r}
\{(\theta_N,\theta_{N+1}),(\theta_{N+1},\theta_{N+2})\}\subset R-r. 
\end{equation}
From the same proposition, we see that $r \subset R$ is the interior of the convex quadrangle, consisting of sections from the segments $p_{a+1}^\#, \tab p_A^\#, \tab f_{b+1}^\#$ and $f_B^\#$, that is, the quadrangle $r$ is obtained from peeling the outer layer of $R$'s subdivisions.\\ 

\noindent When $m=0$, let $(u_0,v_0) := p_A^\# \cap f_{b+1}^\#$ and $(U_0,V_0) := p_{a+1}^\# \cap f_B^\#$. Then
\[(u,v) \in P_A^\# \cap F_b^\# \implies u < u_0 \text{ and } v > v_0\]
and 
\[(u,v) \in P_a^\# \cap F_B^\# \implies u > U_0 \text{ and } v < V_0.\]
When $m=1$, let $(u_1,v_1) := p_A^\# \cap f_B^\#$ and $(U_1,V_1) := p_{a+1}^\# \cap f_{b+1}^\#$. Then
\[(u,v) \in P_a^\# \cap F_b^\# \implies u < u_0 \text{ and } v < v_0\]
and 
\[(u,v) \in P_A^\# \cap F_B^\# \implies u > U_0 \text{ and } v > V_0.\] 
In tandem with formula \eqref{R-r}, this allows us to conclude that, in either case, we have 
\begin{equation}\label{distance}
\operatorname{diam}(r)^2 < d\left((\theta_N,\theta_{N+1}), (\theta_{N+1},\theta_{N+2})\right) = (\theta_N - \theta_{N+1})^2 +  (\theta_{N+1} - \theta_{N+2})^2 < \operatorname{diam}(R)^2,
\end{equation}
where $d$ is the familiar distance formula between two points in $\IE^2$ and $ \operatorname{diam}(R)$ stands for the euclidean diameter of the region $R$.\\ 

\noindent The diameter of the interior of a convex quadrangle in $\IE^2$ with opposite acute angles is the length of the diagonal connecting the vertices corresponding to these acute angles. To verify this simple observation from plane geometry, observe that from the continuity of the euclidean distance formula, the diameter of a convex polygon is either a side or a diagonal. The diagonal connecting the two acute angles will lie opposite to an obtuse angle in either one of the induced triangles obtained. Since for any triangle, the longest side lies opposite to the largest angle, we conclude that this diagonal is longer than all of the four sides in this convex quadrangle. To verify it is longer than the other diagonal, consider one of the four induced triangles obtained after drawing both diagonals and use the same opposite angle argument.\\

\noindent When $m=0$ and $c \in \IZ^+$, the slope of the line containing $p_c^\#$ lie in $(-\infty,-1)$ whereas the slope of the line containing $f_c^\#$ lie in $(-1,0)$. Conclude that both $R$ and $r$ have opposite acute angles, so that their diameters are the length of the diagonal connecting the vertices $p_a^\# \cap f_{B+1}^\#$ with $p_{A+1}^\#\ \cap f_{b}^\#$ and $p_{a+1}^\# \cap f_B^\#$ with $p_A^\#\ \cap f_{b+1}^\#$ respectively. Hence 
\[\operatorname{diam(R)} = \operatorname{d}(p_a^\# \cap f_{B+1}^\#, p_{A+1}^\# \cap f_b^\#)\] 
and
\[\operatorname{diam(r)} = \operatorname{d}(p_{a+1}^\# \cap f_B^\#, p_A^\# \cap f_{b+1}^\#).\]\\

\noindent When $m=1$, the slope of the line containing $p_c^\#$ lie in $(1,\infty)$ whereas the slope of the line containing $f_c^\#$ lie in $(0,1)$. Conclude that both $R$ and $r$ has opposite acute angles, so that their diameters are the length of the diagonal connecting the vertices $p_a^\# \cap f_b^\#$ with $p_{A+1}^\#\ \cap f_{B+1}^\#$ and $p_{a+1}^\# \cap f_{b+1}^\#$ with $p_A^\#\ \cap f_B^\#$ respectively. Hence 
\[\operatorname{diam(R)} = \operatorname{d}(p_a^\# \cap f_b^\#, p_{A+1}^\# \cap f_{B+1}^\#)\] 
and 
\[\operatorname{diam(r)} = \operatorname{d}(p_{a+1}^\# \cap f_{b+1}^\#, p_A^\# \cap f_B^\#).\]\\
In either case, we have 
\[\operatorname{diam(R)}^2 = \left(\frac{B+k+1}{(1-2m)k + (A+k+1)(B+k+1)} - \frac{b+k}{(1-2m)k+ (a+k)(b+k)}\right)^2\] 
\[+ \left(\frac{A+k+1}{(1-2m)k + (A+k+1)(B+k+1)} - \frac{a+k}{(1-2m)k + (a+k)(b+k)}\right)^2\]
and
\[\operatorname{diam(r)}^2 = \left(\frac{B+k}{(1-2m)k + (A+k)(B+k)} - \frac{b+k+1}{(1-2m)k+ (a+k+1)(b+k+1)}\right)^2\] 
\[+ \left(\frac{A+k}{(1-2m)k + (A+k)(B+k)} - \frac{a+k+1}{(1-2m)k + (a+k+1)(b+k+1)}\right)^2.\] 
Formula \eqref{distance} now yields the desired inequalities and the density of the space of Jager Pairs in $\Gamma_{(m,k)}$ establishes their sharpness.\\
\end{proof}

\begin{proof}[Proof (of theorem \ref{main}).]
For all integers $a,b,A$ and $B$ with $l \le a \le A \le L$ and $l \le b \le B \le L$, we have
\[\frac{A+k+1}{(1-2m)k + (b+k)(A+k+1)} \le \frac{L+k+1}{(1-2m)k + (l+k)(L+k+1)}\] 
and
\[\frac{l+k}{(1-2m)k + (l+k)(L+k+1)} \le \frac{a+k}{(1-2m)k + (a+k)(B+k+1)},\]
so that 
\[0 \le \frac{L-l}{(1-2m)k + (l+k+1)(L+k)}\] 
\[\le \frac{A+k+1}{(1-2m)k + (b+k)(A+k+1)} - \dfrac{a+k}{(1-2m)k + (a+k)(B+k+1)}\] 
\[\le \frac{L+k+1}{(1-2m)k + (l+k)(L+k+1)} - \frac{l+k}{(1-2m)k + (l+k)(L+k+1)}\]
\[ = \frac{L-l+1}{(1-2m)k + (l+k)(L+k+1)}.\] 
This argument remains identical after we exchange $a$ for $b$ and $A$ for $B$. Furthermore, setting $a=b:=l$ and $A=B:=L$ shows that we cannot replace these weak inequalities with strict ones. The lemma will now provide us with the result.
\end{proof}

\noindent This theorem will enable us to provide bounds for the rate of growth for the sequence of approximation coefficients between time $n$ and $n+1$, assuming a priori knowledge of bounds for the digits of expansion at times $n+1,\tab n+2$ and $n+3$, as expressed in the following corollary:
\begin{corollary}\label{result}
Assuming the hypothesis of the theorem, we have the inequality
\[ \max\left\{\abs{\theta_{N+1} - \theta_N}, \abs{\theta_{N+2} - \theta_{N+1}} \right\} < \sqrt{2}\tab\left(\frac{L-l+1}{(1-2m)k + (l+k)(L+k+1)}\right).\]
Furthermore, if $L-l>1$, we also have the inequality
\[ \min\left\{\abs{\theta_{N+1} - \theta_N}, \abs{\theta_{N+2} - \theta_{N+1}} \right\} > \sqrt{2}\tab\left(\frac{L-l}{(1-2m)k + (l+k+1)(L+k)}\right).\]
These results extend to the classical $k=1$ cases with one exception. The upper bound statement does not apply for the $m=1$ case when $a_{n+2}=0$ and either $a_{n+1}=0$ or $a_{n+3}=0$.
\end{corollary}

\begin{proof}
When $k>1$, the proof is an immediate consequence of the theorem. Plugging in $m=0$ and letting $k \to 1^+$ establishes the result for the classical Gauss case $m=0, k=1$. When $m=1$, we let $k \to 1^+$ and use the equicontinuity of the family $\{\Psi_{(1,k)}\}_{k>1}$ to obtain the result, once we exclude the possibility of the unbounded region $P_0^\# \cap F_0^\#$ from belonging to the region $R$, when computing the upper bound portion. This happens precisely when either $a_{n+1}=a_{n+2}=0$ or $a_{n+2}=a_{n+3}=0$. 
\end{proof}

\begin{remark}
The bounds given in the last section of the introduction are readily verified once we remember to translate the digits $b_2 = b_3 = 1$ and $b_4 = 3$ in the classical regular continued fraction expansion to $a_2 = a_3 =0$ and $a_4=2$ in the $(0,1)-$expansion.
\end{remark}

\section{Acknowledgments}

This paper is a development of part of the author's Ph.D. dissertation at the University of Connecticut, who has benefited tremendously from the patience and rigor his advisor Andy Haas and would like to thank him for all his efforts. He would also like to extend his gratitude to Alvaro Lozano-Robledo, whose suggestions had contributed to a more clear and elegant finished product.

\end{document}